\documentclass[11pt]{article}

\usepackage[a4paper,left=1.05in,right=1.05in,top=0.85in,bottom=0.9in]{geometry}
\usepackage[T1]{fontenc}
\usepackage{lmodern}
\usepackage{amsmath,amssymb,amsthm}
\usepackage{microtype}
\usepackage{enumitem}
\usepackage{hyperref}
\hypersetup{colorlinks=true,linkcolor=black,citecolor=black,urlcolor=black}

\newtheorem{theorem}{Theorem}[section]
\newtheorem{lemma}[theorem]{Lemma}

\DeclareMathOperator{\ex}{ex}

\setlist[enumerate]{leftmargin=2.1em,itemsep=0.3em,topsep=0.4em}

\usepackage{authblk}

\title{\bf Generalized Tur\'an problems for shorter even cycles}
\author{
Zhen Liu\footnote{Email: 1552580575@qq.com},
~Chuanshu Wu\footnote{Email: cswu97@126.com (Corresponding author)}\\
{\small Center for Discrete Mathematics, Fuzhou University, Fujian, 350003, China}}
\date{\today}

\begin{document}
\maketitle

\begin{abstract}
For graphs \(H\) and \(F\), let \(\ex(n,H,F)\) denote the maximum number of copies of \(H\) in an \(n\)-vertex \(F\)-free graph. Gerbner, Győri, Methuku, and Vizer proved that \(\ex(n,C_{2\ell},C_{2k})=\Theta(n^\ell)\) for \(k>\ell\ge2\). They determined the leading term for \(\ell=2\), but for \(k>\ell\ge3\) their general lower and upper bounds had different leading constants, leaving open the problem of closing this gap. We solve this problem by showing that, for every \(k>\ell\ge2\),
\[
\ex(n,C_{2\ell},C_{2k})
=
\left(\frac{(k-1)_\ell}{2\ell}+o(1)\right)n^\ell,
\]
where $(k-1)_\ell=(k-1)(k-2)\cdots(k-\ell)$.
\end{abstract}

\vspace{1em}

\section{Introduction}

Throughout this paper, all graphs are finite and simple, and copies are
not required to be induced. For a graph $F$, a graph $G$ is called
$F$-free if it contains no copy of $F$ as a subgraph. The classical
Tur\'an number is
\[
\ex(n,F)
=
\max\{|E(G)|:|V(G)|=n\text{ and $G$ is $F$-free}\}.
\]

More generally, for graphs $H$ and $F$, let $N(H,G)$ denote the number
of unlabeled copies of $H$ in $G$, and define
\[
\ex(n,H,F)
=
\max\{N(H,G):|V(G)|=n\text{ and $G$ is $F$-free}\}.
\]
Thus $\ex(n,K_2,F)=\ex(n,F)$, where $K_r$ denotes the complete graph
on $r$ vertices. The clique case $\ex(n,K_s,K_t)$ was determined by
Zykov~\cite{Zykov}, while the systematic study of generalized Tur\'an
numbers for general $H$ and $F$ was initiated by Alon and
Shikhelman~\cite{AlonShikhelman}. For a recent comprehensive survey,
including a discussion of cycle-counting problems, we refer the reader
to Gerbner and Palmer~\cite{GerbnerPalmer}.

For $r\ge3$, let $C_r$ denote the cycle of length $r$. Generalized
Tur\'an problems involving cycles have received considerable attention.
Bollob\'as and Gy\H{o}ri~\cite{BollobasGyori} proved the early result
\[
\ex(n,C_3,C_5)=\Theta(n^{3/2}).
\]
More generally, Gishboliner and Shapira~\cite{GishbolinerShapira}
determined the order of magnitude of $\ex(n,C_r,C_s)$ for all distinct
fixed integers $r,s$ with $r>3$.

We consider the case in which both cycles are even. Throughout the
paper, $k$ and $\ell$ are fixed integers, and $n$ is sufficiently large. For positive
integers $a$ and $t$, write
$
(a)_t=a(a-1)\cdots(a-t+1).
$ The main starting point is the following
result of Gerbner, Gy\H{o}ri, Methuku, and Vizer~\cite{GGMV}.

\begin{theorem}[Gerbner, Gy\H{o}ri, Methuku, and Vizer~\cite{GGMV}]
\label{thm:GGMV-intro}
Let $k>\ell\ge2$. If $\ell=2$, then
\[
\ex(n,C_4,C_{2k})
=
\left(
\frac{(k-1)(k-2)}{4}+o(1)
\right)n^2.
\]
If $\ell\ge3$, then
\[
\left(
\frac{(k-1)_\ell}{2\ell}+o(1)
\right)n^\ell
\le
\ex(n,C_{2\ell},C_{2k})
\le
\left(
\frac{2^{\ell-2}(k-1)^\ell}{2\ell}+o(1)
\right)n^\ell.
\]
\end{theorem}

Thus
\[
\ex(n,C_{2\ell},C_{2k})=\Theta(n^\ell)
\qquad (k>\ell\ge2).
\]
The lower bound above comes from a natural complete bipartite
construction. Let $K_{a,b}$ denote the complete bipartite graph with
parts of sizes $a$ and $b$. Since every $C_{2k}$ in a bipartite graph
uses $k$ vertices from each part, $K_{k-1,n-k+1}$ is $C_{2k}$-free.
Moreover,
\[
N(C_{2\ell},K_{k-1,n-k+1})
=
\frac{(k-1)_\ell(n-k+1)_\ell}{2\ell}
=
\left(
\frac{(k-1)_\ell}{2\ell}+o(1)
\right)n^\ell.
\]
For $k>\ell\ge3$, however, the lower and upper bounds in
Theorem~\ref{thm:GGMV-intro} have different leading constants.
Gerbner et al.\ explicitly asked in the concluding remarks of
\cite{GGMV} whether this gap could be closed.

More precise results are known for the first case $\ell=3$ and $k=4$.
For graphs $H$ and $F$, define
\[
\ex_{\mathrm{bip}}(n,H,F)
=
\max\{N(H,G):|V(G)|=n,\,
G\text{ is bipartite and }F\text{-free}\}.
\]
Gerbner et al.~\cite{GGMV} proved
\[
\ex_{\mathrm{bip}}(n,C_6,C_8)
=
n^3+O(n^{5/2}).
\]
This was strengthened to an exact result by Gy\H{o}ri, He, Lv, Salia,
Tompkins, Varga, and Zhu~\cite{GyoriEtAl}.

\begin{theorem}[Gy\H{o}ri et al.~\cite{GyoriEtAl}]
\label{thm:GyoriEtAl}
For every positive integer $n$,
\[
\ex_{\mathrm{bip}}(n,C_6,C_8)
=
6\binom{n-3}{3}.
\]
Moreover, the unique extremal graph is $K_{3,n-3}$.
\end{theorem}

Recently, Chen and Deng~\cite{ChenDeng} determined the
corresponding unrestricted problem.

\begin{theorem}[Chen and Deng~\cite{ChenDeng}]
\label{thm:ChenDeng-intro}
For all sufficiently large $n$,
\[
\ex(n,C_6,C_8)
=
6\binom{n-3}{3}+12(n-5).
\]
\end{theorem}

In particular,
\[
\ex(n,C_6,C_8)=(1+o(1))n^3,
\]
so the leading constant agrees with the complete bipartite construction.
The leading constant for general $k>\ell\ge3$, however, remained open.

Our main result settles this problem for all shorter even cycles.

\begin{theorem}\label{thm:main}
For every pair of integers $k>\ell\ge2$,
\[
\ex(n,C_{2\ell},C_{2k})
=
\left(
\frac{(k-1)_\ell}{2\ell}+o(1)
\right)n^\ell.
\]
\end{theorem}

Thus $K_{k-1,n-k+1}$ is asymptotically extremal for every
$k>\ell\ge2$. For $k>\ell\ge3$,
Theorem~\ref{thm:main} closes the gap in
Theorem~\ref{thm:GGMV-intro} and determines the previously unknown
leading constant.

\noindent
\textbf{Notation.}
Throughout the paper, let $G=(V(G),E(G))$ be a graph. For
$v\in V(G)$, let $N_G(v)$ and $d_G(v)$ denote its neighborhood and
degree, respectively. For disjoint sets $A,B\subseteq V(G)$, let
$e_G(A,B)$ denote the number of edges with one endpoint in $A$ and
the other in $B$. When the ambient graph is clear, we omit the
subscript $G$.
For $S\subseteq V(G)$, let $G[S]$ denote the subgraph induced by $S$,
and write $e(G)=|E(G)|$. For a graph $J$, we write $J\subseteq G$ if
$V(J)\subseteq V(G)$ and $E(J)\subseteq E(G)$; the subgraph $J$ need
not be induced.
A path of length $r$ is a sequence
$v_0v_1\cdots v_r$ of distinct vertices such that
$v_iv_{i+1}\in E(G)$ for every $0\le i<r$. Thus a path on $r$
vertices has length $r-1$.

For distinct vertices $u,v\in V(G)$, let $f(u,v)=|N_G(u)\cap N_G(v)|$
denote their \emph{codegree}. For $q>0$, let $H_q$ be the graph whose
edges are precisely the pairs
\[
uv\in\binom{V(G)}{2}
\quad\text{such that}\quad
f(u,v)\ge q,
\]
with isolated vertices omitted.

\section{Key lemmas}

We first recall two classical results. The following theorem of
Bondy and Simonovits~\cite{BondySimonovits} gives the standard
extremal bound for even cycles.

\begin{theorem}[Bondy--Simonovits]\label{thm:BS}
For every integer $k\ge2$, every $n$-vertex $C_{2k}$-free graph has
$O_k(n^{1+1/k})$ edges.
\end{theorem}

We also use the following theorem of Erd\H{o}s and
Gallai~\cite{ErdosGallai}.

\begin{theorem}[Erd\H{o}s--Gallai]\label{thm:EG-path}
For every integer $r\ge2$, every $n$-vertex graph containing no path
on $r$ vertices has at most $\frac{r-2}{2}n$ edges.
\end{theorem}

We now establish several properties of the auxiliary graph $H_q$.

\begin{lemma}\label{lem:high-codegree}
Let $k\ge3$, and let $G$ be an $n$-vertex $C_{2k}$-free graph.
For every real $q\ge2k$, the following statements hold.

\begin{enumerate}[label=(\roman*)]
\item
$H_q$ is $C_k$-free.

\item
If $J$ is a subgraph of $H_q$, then
\[
\sum_{uv\in E(J)} f(u,v)
\le
\frac{k-2}{2}\sum_{v\in V(J)} d_G(v)
=
O_k\!\left(n|V(J)|^{1/k}\right).
\]

\item
\[
e(H_q)
=
O_k\!\left(
\left(\frac{n}{q}\right)^{k/(k-1)}
\right).
\]

\item
For every real $p$ with $2k\le p\le n$,
\[
\sum_{\substack{uv\in E(H_{2k})\\ f(u,v)\le p}}
f(u,v)^2
=
O_k\!\left(
n^{k/(k-1)}p^{(k-2)/(k-1)}
\right).
\]
In particular,
\[
\sum_{uv\in E(H_{2k})} f(u,v)^2
=
O_k(n^2).
\]
\end{enumerate}
\end{lemma}

\begin{proof}

\noindent\textit{(i)} Suppose that $H_q$ contains a cycle
$v_1v_2\cdots v_kv_1$, and set $v_{k+1}=v_1$.
For each $1\le i\le k$, choose
$
z_i\in N_G(v_i)\cap N_G(v_{i+1})
$
successively, avoiding $v_1,\ldots,v_k$ and the previously chosen
vertices $z_1,\ldots,z_{i-1}$. At the $i$th step, at most $k+i-1\le 2k-1<q$
vertices are forbidden. Since $v_iv_{i+1}\in E(H_q)$, we have
$f(v_i,v_{i+1})\ge q$, so such a choice is always possible.
Thus all vertices
$
v_1,\ldots,v_k,z_1,\ldots,z_k
$
are distinct, and
$
v_1z_1v_2z_2\cdots v_kz_kv_1
$
is a $C_{2k}$ in $G$, a contradiction. Hence $H_q$ is $C_k$-free.

\medskip
\noindent\textit{(ii)}
Fix $x\in V(G)$. We first show that
$J[N_G(x)\cap V(J)]$ contains no path on $k$ vertices. Suppose that
$z_1z_2\cdots z_k$ is such a path. Since
$z_iz_{i+1}\in E(J)\subseteq E(H_q)$, we have
$f(z_i,z_{i+1})\ge q$ for every $1\le i\le k-1$.

Choose
$y_i\in N_G(z_i)\cap N_G(z_{i+1})$ successively for
$1\le i\le k-1$, avoiding $x$, the vertices $z_1,\ldots,z_k$, and
the previously chosen vertices $y_1,\ldots,y_{i-1}$. At the $i$th
step, at most $k+i\le2k-1<q$ vertices are forbidden, so such a choice
is always possible. Since $z_1,\ldots,z_k\in N_G(x)$, the vertices in
$xz_1y_1z_2y_2\cdots z_{k-1}y_{k-1}z_kx$ are distinct and form a
$C_{2k}$ in $G$, a contradiction. Hence
$J[N_G(x)\cap V(J)]$ contains no path on $k$ vertices.

By Theorem~\ref{thm:EG-path},
\[
e\bigl(J[N_G(x)\cap V(J)]\bigr)
\le
\frac{k-2}{2}|N_G(x)\cap V(J)|.
\]
Summing over $x\in V(G)$ gives
\begin{align*}
\sum_{uv\in E(J)}f(u,v)
&=
\sum_{x\in V(G)}
e\bigl(J[N_G(x)\cap V(J)]\bigr)\\
&\le
\frac{k-2}{2}
\sum_{x\in V(G)}|N_G(x)\cap V(J)|
=
\frac{k-2}{2}
\sum_{v\in V(J)}d_G(v).
\end{align*}

It remains to show that
$\sum_{v\in V(J)}d_G(v)
=O_k\!\left(n|V(J)|^{1/k}\right)$.
This is immediate if $V(J)=\varnothing$. If $|V(J)|\ge n/2$, then
Theorem~\ref{thm:BS} gives
\[
\sum_{v\in V(J)}d_G(v)
\le
2e(G)
=
O_k(n^{1+1/k})
=
O_k\!\left(n|V(J)|^{1/k}\right).
\]

Now suppose that $1\le |V(J)|<n/2$, and set
$r=\left\lceil (n-|V(J)|)/|V(J)|\right\rceil$.
Partition $V(G)\setminus V(J)$ into
$R_1\cup\cdots\cup R_r$ so that the sizes of any two parts differ by
at most one. Then $|R_i|\le |V(J)|$ for every $i\in \{1,\cdots ,r\}$, and hence
$G[V(J)\cup R_i]$ has at most $2|V(J)|$ vertices. Since it is
$C_{2k}$-free, Theorem~\ref{thm:BS} gives
\[
e\bigl(G[V(J)\cup R_i]\bigr)
=
O_k\!\left(|V(J)|^{1+1/k}\right)
\qquad (1\le i\le r).
\]

Since $r\ge2$, every edge of $G[V(J)]$ is counted $r$ times in
$\sum_{i=1}^r e(G[V(J)\cup R_i])$, while every edge between
$V(J)$ and $V(G)\setminus V(J)$ is counted once. Therefore
\begin{align*}
\sum_{v\in V(J)}d_G(v)
&=
2e(G[V(J)])
+
e(V(J),V(G)\setminus V(J))\\
&\le
\sum_{i=1}^r e\bigl(G[V(J)\cup R_i]\bigr)=
O_k\!\left(r|V(J)|^{1+1/k}\right)
=
O_k\!\left(n|V(J)|^{1/k}\right).
\end{align*}
This proves~(ii).

\medskip
\noindent\textit{(iii)}
If $e(H_q)=0$, there is nothing to prove. Otherwise, since $H_q$
has no isolated vertices, we have
$|V(H_q)|\le 2e(H_q)$. By the definition of $H_q$ and part~(ii),
\[
q e(H_q)
\le
\sum_{uv\in E(H_q)} f(u,v)
=
O_k\!\left(n|V(H_q)|^{1/k}\right)
=
O_k\!\left(ne(H_q)^{1/k}\right).
\]
It follows that
\[
e(H_q)
=
O_k\!\left(
\left(\frac{n}{q}\right)^{k/(k-1)}
\right).
\]

\medskip
\noindent\textit{(iv)}
For each edge \(uv\in E(H_{2k})\) with \(f(u,v)\le p\), there is a unique integer \(j\ge0\) such that
$
2^{j+1}k\le f(u,v)<2^{j+2}k.
$
Necessarily \(2^{j+1}k\le p\). For a fixed \(j\), all edges in this range belong to \(H_{2^{j+1}k}\), so their number is at most \(e(H_{2^{j+1}k})\). Moreover, for each such edge,
$
f(u,v)^2<4(2^{j+1}k)^2.
$ Hence, by part~(iii),
\begin{align*}
\sum_{\substack{uv\in E(H_{2k})\\ f(u,v)\le p}} f(u,v)^2
&\le
4\sum_{\substack{j\ge0\\2^{j+1}k\le p}}
(2^{j+1}k)^2 e\bigl(H_{2^{j+1}k}\bigr)\\
&=
O_k\!\left(
n^{k/(k-1)}
\sum_{\substack{j\ge0\\2^{j+1}k\le p}}
(2^{j+1}k)^{(k-2)/(k-1)}
\right)\\
&=
O_k\!\left(
n^{k/(k-1)}p^{(k-2)/(k-1)}
\right),
\end{align*}
where the last step follows from the geometric sum.

Finally, $f(u,v)\le n$ for all distinct $u,v\in V(G)$. Taking $p=n$
therefore gives
\[
\sum_{uv\in E(H_{2k})} f(u,v)^2
=
O_k(n^2).
\]
\end{proof}

\begin{lemma}\label{lem:cliques}
Let $k>\ell\ge3$, let $F$ be a $C_k$-free graph, and let
$Q_1,\ldots,Q_p$ be cliques of $F$, not necessarily distinct; empty
cliques are allowed. For each $uv\in E(F)$, define
\[
c(u,v)
=
\bigl|\{a\in\{1,\ldots,p\}:\{u,v\}\subseteq Q_a\}\bigr|.
\]
Then
\[
\sum_{C_\ell\subseteq F}
\prod_{uv\in E(C_\ell)}c(u,v)
\le
\frac{(k-1)_\ell}{2\ell}p^\ell.
\]
\end{lemma}

\begin{proof}
Fix $(a_1,\ldots,a_\ell)\in\{1,\ldots,p\}^\ell$, with repetitions
allowed. An ordered $\ell$-tuple of distinct vertices
$(v_1,\ldots,v_\ell)$ is called \emph{compatible} with
$(a_1,\ldots,a_\ell)$ if
$\{v_j,v_{j+1}\}\subseteq Q_{a_j}$ for every $1\le j\le\ell$, where
$v_{\ell+1}=v_1$. Since each $Q_{a_j}$ is a clique, we have
$v_jv_{j+1}\in E(F)$ for every $1\le j\le\ell$. Hence
$v_1v_2\cdots v_\ell v_1$ is a copy of $C_\ell$ in $F$.

Set $a_0=a_\ell$ and
\[
U=\bigcup_{j=1}^{\ell}
\bigl(Q_{a_{j-1}}\cap Q_{a_j}\bigr).
\]
Every compatible tuple satisfies
$v_j\in Q_{a_{j-1}}\cap Q_{a_j}$ for $1\le j\le\ell$, and hence
$\{v_1,\ldots,v_\ell\}\subseteq U$.

We claim that $|U|\le k-1$ whenever a compatible tuple exists.
Suppose instead that $|U|\ge k$, and fix a compatible tuple
$(v_1,\ldots,v_\ell)$. Choose
$W\subseteq U\setminus\{v_1,\ldots,v_\ell\}$ with $|W|=k-\ell$.
Since \(W\subseteq U\), for every \(w\in W\) there exists an index
\(j\in\{1,\ldots,\ell\}\) such that
$
w\in Q_{a_{j-1}}\cap Q_{a_j}.
$
Choose one such index for each \(w\in W\), and insert \(w\) between
\(v_j\) and \(v_{j+1}\), grouping the vertices assigned to the same
index in any order. Since all vertices inserted between \(v_j\) and
\(v_{j+1}\), together with \(v_j\) and \(v_{j+1}\), belong to
\(Q_{a_j}\), all consecutive vertices in the resulting cyclic sequence
are adjacent. Hence the resulting sequence forms a \(C_k\) in \(F\),
a contradiction.

Thus, for each fixed $(a_1,\ldots,a_\ell)$, there are at most
$(|U|)_\ell\le(k-1)_\ell$ compatible tuples. Since there are
$p^\ell$ choices for $(a_1,\ldots,a_\ell)$, the total number of
compatible pairs
$$
\bigl((a_1,\ldots,a_\ell),(v_1,\ldots,v_\ell)\bigr)
$$
is at most $p^\ell(k-1)_\ell$.

On the other hand, fix a copy $C_\ell\subseteq F$, choose a starting
vertex and a direction around the cycle, and write its vertices in
order as $v_1,\ldots,v_\ell$, with $v_{\ell+1}=v_1$. For each
$1\le j\le\ell$, there are exactly $c(v_j,v_{j+1})$ choices for
$a_j$, independently of the other choices. Hence the number of
compatible sequences $(a_1,\ldots,a_\ell)$ corresponding to this
choice of starting vertex and direction is
$
\prod_{j=1}^{\ell}c(v_j,v_{j+1}).
$

Each copy of $C_\ell$ has $\ell$ choices for the starting vertex and
two choices for the direction. Therefore the total number of compatible
pairs is
\[
2\ell
\sum_{C_\ell\subseteq F}
\prod_{uv\in E(C_\ell)}c(u,v).
\]
Comparing this with the upper bound $p^\ell(k-1)_\ell$ obtained above
gives
\[
2\ell
\sum_{C_\ell\subseteq F}
\prod_{uv\in E(C_\ell)}c(u,v)
\le
p^\ell(k-1)_\ell,
\]
and the result follows.
\end{proof}

\section{Proof of the main theorem}
Let $G$ be an $n$-vertex $C_{2k}$-free graph. For distinct vertices
$u,v\in V(G)$, call the pair $uv$ \emph{fat} if $f(u,v)\ge 2k,$
and \emph{non-fat} otherwise.
Write a $2\ell$-cycle as
$v_0w_0v_1w_1\cdots v_{\ell-1}w_{\ell-1}v_0$, and set
$v_\ell=v_0$ and $w_\ell=w_0$. We call
$v_0v_1\cdots v_{\ell-1}v_0$ and
$w_0w_1\cdots w_{\ell-1}w_0$ its two \emph{alternating sides}.
These alternating sides need not be cycles in $G$. More generally,
an alternating side
$x_0x_1\cdots x_{\ell-1}x_0$ is called \emph{fat} if
$f(x_i,x_{i+1})\ge 2k$ for every $0\le i\le\ell-1,$
where $x_\ell=x_0$.

\begin{lemma}\label{lem:nonfat}
The number of $2\ell$-cycles for which each alternating side contains
a non-fat pair is $o(n^\ell)$.
\end{lemma}

\begin{proof}
Let $m=e(G)$. We first record a simple path bound. For every integer
$t\ge1$, the number of paths of length $2t-1$ in $G$ is at most
$(2m)^t$. Indeed, if
$u_1u_2\cdots u_{2t}$ is such a path, then for each $1\le i\le t$, there are at most $2m$ choices for
$(u_{2i-1},u_{2i})$, corresponding to the choice of an edge of $G$
and an order of its endpoints. Ignoring the remaining
adjacency and distinctness conditions gives the required bound.

We divide the cycles counted by the lemma into two cases.

\medskip
\noindent\textbf{Case 1.}
At least one alternating side contains both fat and non-fat pairs.

Along such a side, a fat pair is followed by a non-fat pair. After
interchanging the two alternating sides if necessary and cyclically
relabeling the vertices, we may assume that $v_{\ell-1}v_0$ is fat
and $v_0v_1$ is non-fat.

Choose the path
$v_1w_1v_2w_2\cdots v_{\ell-2}w_{\ell-2}$.
Since its length is $2\ell-5$, there are at most
$(2m)^{\ell-2}$ choices. For an upper bound, we ignore the condition
$w_{\ell-2}v_{\ell-1}\in E(G)$.
Since $v_{\ell-1}v_0$ is fat, we have
$v_{\ell-1}v_0\in E(H_{2k})$, while $w_{\ell-1}$ is a common
neighbor of $v_{\ell-1}$ and $v_0$. Hence the number of possible
triples $(v_{\ell-1},v_0,w_{\ell-1})$ is at most
$2\sum_{uv\in E(H_{2k})}f(u,v)$. By
Lemma~\ref{lem:high-codegree}(ii),
\[
2\sum_{uv\in E(H_{2k})}f(u,v)
\le
(k-2)\sum_{u\in V(H_{2k})}d_G(u)
\le
2(k-2)m.
\]
Since $v_0v_1$ is non-fat, there are fewer than $2k$ choices for
$w_0\in N_G(v_0)\cap N_G(v_1)$.

These choices determine all vertices of the cycle. Therefore the
number of cycles in Case~1 is at most
\[
(2m)^{\ell-2}\cdot2(k-2)m\cdot2k
=
O_{k,\ell}(m^{\ell-1}).
\]

\medskip
\noindent\textbf{Case 2.}
Both alternating sides consist entirely of non-fat pairs.

Fix a cyclic labeling
$v_0w_0v_1w_1\cdots v_{\ell-1}w_{\ell-1}v_0$ of the cycle.
After deleting $w_0$ and $v_{\ell-1}$, the remaining cycle edges form
the paths
$v_1w_1v_2w_2\cdots v_{\ell-2}w_{\ell-2}$ and
$w_{\ell-1}v_0$, of lengths $2\ell-5$ and $1$, respectively.
By the path bound above, these two paths can be chosen in at most
$(2m)^{\ell-2}(2m)=(2m)^{\ell-1}$ ways.

Since $v_0v_1$ and $w_{\ell-2}w_{\ell-1}$ are non-fat, there are
fewer than $2k$ choices for
$w_0\in N_G(v_0)\cap N_G(v_1)$ and fewer than $2k$ choices for
$v_{\ell-1}\in N_G(w_{\ell-2})\cap N_G(w_{\ell-1})$.
These choices determine all vertices of the cycle. Therefore the
number of cycles in Case~2 is at most
\[
(2m)^{\ell-1}(2k)^2
=
O_{k,\ell}(m^{\ell-1}).
\]

Combining the two cases, the number of cycles in question is
$O_{k,\ell}(m^{\ell-1})$. Since $G$ is $C_{2k}$-free,
Theorem~\ref{thm:BS} gives $m=O_k(n^{1+1/k})$. As $\ell<k$,
$
(\ell-1)\left(1+\frac1k\right)<\ell,
$
and therefore $m^{\ell-1}=o(n^\ell)$. This proves the lemma.
\end{proof}

For a real square matrix $X=(X_{ij})$, let
\[
\operatorname{tr}(X)=\sum_i X_{ii},
\qquad
\|X\|_{\mathrm F}
=
\left(\sum_{i,j}X_{ij}^2\right)^{1/2}
\]
denote its trace and Frobenius norm, respectively. We use the following
standard inequalities; see Ford~\cite[Theorems~6.5 and~7.5]{Ford}.

\begin{theorem}\label{thm:matrix}
Let $X$ and $Y$ be real square matrices of the same order. Then
\[
|\operatorname{tr}(XY)|
\le
\|X\|_{\mathrm F}\|Y\|_{\mathrm F},
\qquad
\|XY\|_{\mathrm F}
\le
\|X\|_{\mathrm F}\|Y\|_{\mathrm F}.
\]
\end{theorem}

\begin{lemma}\label{lem:weighted}
Let $k>\ell\ge3$, and let $G$ be an $n$-vertex $C_{2k}$-free graph.
Then
\[
\sum_{C_\ell\subseteq H_{2k}}
\prod_{uv\in E(C_\ell)}f(u,v)
\le
\left(
\frac{(k-1)_\ell}{2\ell}+o(1)
\right)n^\ell.
\]
\end{lemma}

\begin{proof}
Let $q=n/\log n$. For all sufficiently large $n$, we have
$2k\le q\le n$.

Define two symmetric matrices
$A=(a_{uv})_{u,v\in V(G)}$ and
$D=(d_{uv})_{u,v\in V(G)}$ by setting $a_{uu}=d_{uu}=0$ and, for
distinct $u,v\in V(G)$,
\[
a_{uv}=
\begin{cases}
f(u,v),& f(u,v)\ge2k,\\
0,& f(u,v)<2k,
\end{cases}
\qquad
d_{uv}=
\begin{cases}
f(u,v),& 2k\le f(u,v)<q,\\
0,& \text{otherwise}.
\end{cases}
\]

By Lemma~\ref{lem:high-codegree}(iv),
\[
\|A\|_{\mathrm F}^2
=
2\sum_{uv\in E(H_{2k})}f(u,v)^2
=
O_k(n^2).
\]
Applying the same estimate with $p=q$ gives
\[
\|D\|_{\mathrm F}^2
=
2\sum_{\substack{uv\in E(H_{2k})\\ f(u,v)<q}}f(u,v)^2
=
O_k\!\left(
n^{k/(k-1)}q^{(k-2)/(k-1)}
\right)
=
o(n^2).
\]
Thus $\|A\|_{\mathrm F}=O_k(n)$ and
$\|D\|_{\mathrm F}=o(n)$. The factor $2$ above comes from the two
symmetric matrix entries corresponding to each edge.

Let $\mathcal C$ be the family of copies
$C_\ell\subseteq H_{2k}$ that are not contained in
$H_{2k}[V(H_q)]$. For each $C\in\mathcal C$, choose a vertex
$x\in V(C)\setminus V(H_q)$. Since isolated vertices are omitted from
$H_q$, no pair $xy\in E(C)$ incident with $x$ belongs to $E(H_q)$.
As $C\subseteq H_{2k}$, such a pair satisfies
\[
2k\le f(x,y)<q.
\]
Choose one of the two edges of $C$ incident with $x$, say $xy$.
Consider the two possible orders of the endpoints of $xy$. For each
order, write the vertices of $C$ consecutively as
$x_1,\ldots,x_\ell$, where $x_{\ell+1}=x_1$ and
$\{x_1,x_2\}=\{x,y\}$. Then
$d_{x_1x_2}=f(x_1,x_2)$, while
$a_{x_ix_{i+1}}=f(x_i,x_{i+1})$ for every $2\le i\le\ell$.
Therefore
\[
2\sum_{C\in\mathcal C}
\prod_{uv\in E(C)}f(u,v)
\le
\sum_{x_1,\ldots,x_\ell\in V(G)}
d_{x_1x_2}
\prod_{i=2}^{\ell}a_{x_ix_{i+1}}
=
\operatorname{tr}(DA^{\ell-1}).
\]
By Theorem~\ref{thm:matrix},
\[
\sum_{C\in\mathcal C}
\prod_{uv\in E(C)}f(u,v)
\le
\frac12\|D\|_{\mathrm F}\|A^{\ell-1}\|_{\mathrm F}
\le
\frac12\|D\|_{\mathrm F}\|A\|_{\mathrm F}^{\ell-1}
=
o(n^\ell).
\]
Hence
\begin{equation}\label{eq:truncation}
\sum_{C_\ell\subseteq H_{2k}}
\prod_{uv\in E(C_\ell)}f(u,v)
=
\sum_{C_\ell\subseteq H_{2k}[V(H_q)]}
\prod_{uv\in E(C_\ell)}f(u,v)
+
o(n^\ell).
\end{equation}

It remains to estimate the sum over
$H_{2k}[V(H_q)]$. Since $H_q$ has no isolated vertices,
Lemma~\ref{lem:high-codegree}(iii) gives
\begin{equation}\label{eq:Vq}
|V(H_q)|
\le
2e(H_q)
=
O_k\!\left((\log n)^{k/(k-1)}\right).
\end{equation}

For each $x\in V(G)$, let
\[
R_x=N_G(x)\cap V(H_q),
\]
and let $B$ be the set of vertices $x$ for which $R_x$ is not a clique
in $H_{2k}$, where the empty set and a single vertex are also regarded
as cliques. For each $x\in B$, choose distinct
$u_x,v_x\in R_x$ such that $u_xv_x\notin E(H_{2k})$. Then
$f(u_x,v_x)<2k$, while $x\in N_G(u_x)\cap N_G(v_x)$.

For a fixed pair $\{u,v\}\subseteq V(H_q)$, there are exactly
$f(u,v)$ common neighbors of $u$ and $v$. Hence
\[
|B|
\le
\sum_{\substack{\{u,v\}\in\binom{V(H_q)}{2}\\ f(u,v)<2k}}
f(u,v)
\le
(2k-1)\binom{|V(H_q)|}{2}
=
O_k\!\left(|V(H_q)|^2\right).
\]

We now construct an indexed family of cliques in
$H_{2k}[V(H_q)]$. For each $x\notin B$, include the clique $R_x$.
For each $x\in B$ and each edge
$uv\in E(H_{2k}[R_x])$, include the two-vertex clique $\{u,v\}$,
labeled by $(x,uv)$. Cliques with different labels are regarded as
different members of the family. Write the resulting family as
$Q_1,\ldots,Q_p$.
We have
\[
p
=
n-|B|
+
\sum_{x\in B}e\bigl(H_{2k}[R_x]\bigr).
\]
Since
$e(H_{2k}[R_x])\le\binom{|V(H_q)|}{2}$ for every $x\in B$, the
preceding bound on $|B|$ and \eqref{eq:Vq} give
\begin{align*}
|p-n|
\le
|B|
+
\sum_{x\in B}e\bigl(H_{2k}[R_x]\bigr)\le
|B|
+
|B|\binom{|V(H_q)|}{2}=
O_k\!\left(|V(H_q)|^4\right)
=
o(n).
\end{align*}
Thus
\begin{equation}\label{eq:p-asymptotic}
p=n+o(n).
\end{equation}

Fix \(uv\in E(H_{2k}[V(H_q)])\). We claim that
\[
\bigl|\{a\in\{1,\ldots,p\}:u,v\in Q_a\}\bigr|
=
f(u,v).
\]
Indeed, let \(x\in N_G(u)\cap N_G(v)\). Then \(u,v\in R_x\). If
\(x\notin B\), the clique \(R_x\) is included in the family and
contains both \(u\) and \(v\). If \(x\in B\), then
\(uv\in E(H_{2k}[R_x])\), so the clique \(\{u,v\}\) labeled by
\((x,uv)\) is included. Thus each common neighbor of \(u\) and \(v\)
gives exactly one member of the family containing both vertices.
Conversely, every such member arises from a common neighbor of
\(u\) and \(v\). This proves the claim.

By Lemma~\ref{lem:high-codegree}(i), \(H_{2k}\) is \(C_k\)-free, and
hence so is \(H_{2k}[V(H_q)]\). Therefore Lemma~\ref{lem:cliques},
together with the claim above and \eqref{eq:p-asymptotic}, gives
\[
\sum_{C_\ell\subseteq H_{2k}[V(H_q)]}
\prod_{uv\in E(C_\ell)}f(u,v)
\le
\frac{(k-1)_\ell}{2\ell}p^\ell
=
\left(
\frac{(k-1)_\ell}{2\ell}+o(1)
\right)n^\ell.
\]
Combining this with \eqref{eq:truncation} proves the lemma.
\end{proof}

\begin{proof}[Proof of Theorem~\ref{thm:main}]
The lower bound follows from the construction
$K_{k-1,n-k+1}$ of Gerbner et al.~\cite{GGMV}, as recalled in the
Introduction.

For the upper bound, the case $\ell=2$ follows from~\cite{GGMV}.
Assume that $k>\ell\ge3$, and let $G$ be an $n$-vertex
$C_{2k}$-free graph. By Lemma~\ref{lem:nonfat}, all but
$o(n^\ell)$ copies of $C_{2\ell}$ in $G$ have a fat alternating side.

Choose one fat alternating side from each such cycle. Fix a selected
side $C_\ell\subseteq H_{2k}$. For each $uv\in E(C_\ell)$, the
opposite alternating side contains a common neighbor of $u$ and $v$.
Hence the number of cycles assigned to $C_\ell$ is at most
$
\prod_{uv\in E(C_\ell)}f(u,v),
$
where choices with repeated vertices or not forming a cycle may also
be counted.
Consequently,
\begin{align*}
N(C_{2\ell},G)
\le
o(n^\ell)
+
\sum_{C_\ell\subseteq H_{2k}}
\prod_{uv\in E(C_\ell)}f(u,v)\le
\left(
\frac{(k-1)_\ell}{2\ell}+o(1)
\right)n^\ell,
\end{align*}
where the last inequality follows from Lemma~\ref{lem:weighted}.
This completes the proof.
\end{proof}

\section*{Declaration on the Use of Generative AI}
The authors used ChatGPT 5.6 Pro to assist in discussing proof strategies, checking proofs, and improving exposition.

\end{document}